\documentclass{amsart}
\usepackage[utf8]{inputenc}
\usepackage{amsmath}
\usepackage{amsfonts}
\usepackage{amssymb}
\usepackage{float}
\usepackage{tikz}
\usepackage[colorlinks=true]{hyperref}
\hypersetup{urlcolor=blue, citecolor=red}
\usepackage{hyperref}

\newtheorem{theorem}{Theorem}[section]
\newtheorem{corollary}[theorem]{Corollary}

\newtheorem{lemma}[theorem]{Lemma}
\newtheorem{proposition}[theorem]{Proposition}

\theoremstyle{definition}

\newtheorem{remark}[theorem]{Remark}

\newcommand{\R}{\mathbb{R}}
\newcommand{\N}{\mathbb{N}}

\newcommand{\T}{\mathbb{T}}
\newcommand{\Z}{\mathbb{Z}}

\begin{document}

\title[Quasi-periodic quasilinear equations]{Low-regularity well-posedness for dispersive equations with derivative nonlinearity and quasi-periodic initial data}

\author{Robert Schippa}
\email{rschippa@uni-bonn.de}
\address{Mathematical Institute of the University of Bonn, Endenicher Allee 60, D-53115 Bonn, Germany}


\makeatletter
\@namedef{subjclassname@2020}{%
  \textup{2020} Mathematics Subject Classification}
\makeatother


\begin{abstract}
We show low-regularity well-posedness of the Korteweg-de Vries equation with spatially quasi-periodic initial data. To this end, we employ frequency-dependent time localization and a bilinear version of the C\'ordoba--Fefferman square function estimate to show a bilinear Strichartz estimate for quasi-periodic functions. The solutions are proved to preserve the Sobolev regularity of the initial data, which was not the case in earlier works. The argument extends to other dispersion relations and higher order nonlinearities.
\end{abstract}

\maketitle

\section{Introduction}

The low-regularity well-posedness of nonlinear dispersive equations with quasi-periodic initial data has recently undergone a flurry of developments. The works of Damanik \emph{et al.} and Xu \cite{DamanikLiXu2024,Xu2025} (see also Papenburg \cite{Papenburg2025GeneralLWP}) rely on estimates for Picard iterates or energy arguments to prove well-posedness results, which do not take advantage of dispersive effects. In recent work of the author \cite{Schippa2025Quasiperiodic} it was pointed out how square function and decoupling arguments yield sharp Strichartz estimates and local well-posedness results for semilinear dispersive equations with quasi-periodic initial data. By these means, the sharp (up to endpoints) local well-posedness of the cubic NLS was obtained in \cite{Schippa2025Quasiperiodic}. Furthermore, in the works \cite{DamanikLiXu2024,Xu2025} the obtained solutions were constructed with weaker Fourier decay than the initial data. The solutions constructed in \cite{Schippa2025Quasiperiodic} preserved the initial regularity.

\smallskip

Here we consider dispersive PDE with quasi-periodic initial data and derivative nonlinearity. We focus our attention on Sobolev spaces based on the Besicovitch-$2$-almost periodic functions, which are limits of trigonometric polynomials on the real line with respect to a spatially averaged $L^2$-norm.
We analyze the KdV equation in detail, which reads
\begin{equation}
\label{eq:KdVIntro}
\left\{ \begin{array}{cl}
\partial_t u + \partial_x^3 u &= u \partial_x u, \quad (t,x) \in \R \times \R, \\
u(0) &= u_0 \in H^s_{\Lambda}(\R).
\end{array} \right.
\end{equation}
We consider real-valued solutions. The argument applies more generally to dispersive equations with different dispersion relation and general power-type derivative nonlinearities. Examples include the (generalized) Benjamin-Ono equation and the generalized KdV equation. We consider the Benjamin-Ono equation below.
Function spaces for quasi-periodic functions are explained in detail in Section \ref{section:Prelim}. 

The KdV equation has been extensively studied, also for quasi-periodic initial data. In \cite{BinderDamanikGoldsteinLukic2018} global solutions were constructed for special quasi-periodic analytic initial data, verifying the Deift conjecture for this class of initial data. The Deift conjecture claims global existence of almost-periodic in time solutions for spatially almost-periodic initial data. This was disproved in its full generality in \cite{ChapoutoKillipVisan2024}.

The Sobolev spaces under consideration $H^s_{\Lambda}$ for quasi-periodic functions with spectrum in $\Lambda= \omega_1 \Z + \ldots + \omega_{\nu} \Z$, $\omega \in \R_{>0}^{\nu}$ being rationally independent, are isomorphic to $H^s(\T^{\nu})$ via the linear mapping $H^s(\T^{\nu}) \to H^s_{\Lambda}$ given by $f \mapsto f(\omega \cdot x)$ for $s \geq 0$. Let $\partial_{\omega} := \omega \cdot \nabla $. The isomorphism transforms the equation with quasi-periodic data to the following degenerate KdV equation on $\T^{\nu}$:
\begin{equation}
\label{eq:DegenerateKdVIntro}
\left\{ \begin{array}{cl}
\partial_t u + \partial_{\omega}^3 u &= u \partial_{\omega} u, \quad (t,x) \in \R \times \T^{\nu}, \\
u(0) &= u_0 \in H^s(\T^{\nu}).
\end{array} \right.
\end{equation}
By energy arguments, this equation is locally well-posed for $s>1+\frac{\nu}{2}$. For the Benjamin-Ono equation this argument is detailed by Aitzhan--Ambrose in \cite{AitzhanAmbrose2024}.

 The constructed solutions in \cite{BinderDamanikGoldsteinLukic2018} were not proved to satisfy the same exponential Fourier decay as the initial data. Tsugawa \cite{Tsugawa2012} proved local well-posedness via the contraction mapping principle, imposing a weight on the low frequencies. He moreover argued \cite[Proposition~5.1,~Remark~5.2]{Tsugawa2012} that with the outset of \cite{Tsugawa2012} solving the KdV initial value problem for quasi-periodic data is not possible via the contraction mapping principle without imposing a low frequency weight. 
The following is our main result:
\begin{theorem}
\eqref{eq:KdVIntro} is locally well-posed for initial data in $H^{s}_{\Lambda}$ for $s>1+\frac{\nu-1}{2} $.
\end{theorem} 
This appears to be the first local well-posedness result for the KdV equation with quasi-periodic initial data, in which the solutions are proved to be as regular as the initial data without weight on the low frequencies, and which does not follow from energy arguments. Local well-posedness is understood in the Hadamard sense, i.e., existence, uniqueness, and continuous dependence of the solutions on the initial data, which is here the case in the regularity of the initial data.

\smallskip

In the present paper we show a novel low-regularity well-posedness result using dispersive properties. However, we cannot expect to solve the KdV equation via the contraction mapping principle - as was the case for NLS \cite{Schippa2025Quasiperiodic}. 
The reason is the absence of smoothing, already in the quasi-periodic setting and resonant frequencies, which are absent in the periodic setting. We remark that endpoint and long-time Strichartz estimates were recently reported in \cite{Inami2026}.

Bourgain \cite{Bourgain1993KdV} showed how to overcome the derivative loss for the KdV equation on the torus by resonance considerations in Fourier restriction spaces. The resonance function reads
\begin{equation*}
\Omega_{\T}(k_1,k_2) = (k_1+k_2)^3 -k_1^3 - k_2^3 = 3 k_1 k_2(k_1+k_2)
\end{equation*}
and for $|k_1| \geq 2|k_2|$ we have $|\Omega_{\T}(k_1,k_2)| \sim |k_1|^2 |k_2|$. Fourier restriction norms allow to recover a factor of $|\Omega_{\T}|^{-\frac{1}{2}}$, which mitigates the derivative loss, and allowed for solving the KdV equation via the contraction mapping in Sobolev spaces $H^{s}(\T)$, $s\geq - \frac{1}{2}$ (cf. \cite[Theorem~1.8]{KenigPonceVega1996}).
The analysis in the genuine quasi-periodic case sharply diverges from the periodic case, as the resonance function can become arbitrarily small for high frequencies:
\begin{equation*}
\Omega(k_1,k_2) = ((k_1+k_2) \cdot \omega)^3 - (k_1 \cdot \omega)^3 - (k_2 \cdot \omega)^3 = 3( (k_1+k_2)\cdot \omega) (k_1 \cdot \omega) (k_2 \cdot \omega).
\end{equation*}
Carrying out Bourgain's argument for quasi-periodic functions, Tsugawa incurred the low frequency weight.
Here we use frequency-dependent time localization to ameliorate the derivative loss. This is a robust approach to treat quasilinear dispersive equations, as well on the real line as on tori - as the behavior of linear solutions after frequency-dependent time localization becomes independent of the ambient geometry. The seminal work \cite{IonescuKenigTataru2008} combined Fourier restriction and frequency-dependent time localization to show a quasilinear local well-posedness result; we also refer to  \cite{BurqGerardTzvetkov2004,KochTzvetkov2003,KochTataru2007} with the list of references being merely exemplary. 

\smallskip

The versatility of frequency-dependent time localization to treat quasilinear equations on non-Euclidean geometries was emphasized in the PhD thesis of the author \cite{ShorttimeFourierTransformRestriction}. The present work is the first treating quasi-periodic problems with frequency-dependent time localization. We extend the arguments in \cite{Schippa2025Quasiperiodic,ShorttimeFourierTransformRestriction}. 

For periodic solutions to the Airy equation with separated frequencies $N_2 \leq N_1/8$, the bilinear Strichartz estimate on frequency-dependent times, which resembles the estimate on the real line, reads
\begin{equation*}
\| e^{t \partial_x^3} P_{N_1} u_1 e^{t \partial_x^3} P_{N_2} u_2 \|_{L^2_{t,x}([0,N_1^{-2}] \times \T)} \lesssim N_1^{-1} \| u_1 \|_{L^2(\T)} \| u_2 \|_{L^2(\T)}.
\end{equation*}
This is covered in the paper by Moyua--Vega \cite{MoyuaVega2008}. 

Estimates on Euclidean space are easier to obtain, Bourgain \cite{Bourgain1998} recorded an instance for the Schr\"odinger equation.
Then, Hani \cite{Hani2012} extended the estimate to general manifolds. In \cite{Schippa2025Trilinear} the Schr\"odinger estimate was recovered by $\ell^2$-decoupling arguments. We shall use the latter to extend the argument to quasi-periodic functions and other dispersion relations as well. The cornerstone is a bilinear refinement of the C\'ordoba--Fefferman square function estimate \cite{Fefferman1973,Cordoba1979,Cordoba1982}.

\smallskip

As remarked earlier, the method applies to the Benjamin-Ono equation as well:
\begin{equation}
\label{eq:BOIntroduction}
\left\{ \begin{array}{cl}
\partial_t u + \mathcal{H} \partial_x^2 u &= u \partial_x u, \quad (t,x) \in \R \times \R, \\
u(0) &= u_0 \in H_{\Lambda}^s(\R).
\end{array} \right.
\end{equation}
Energy arguments yield again local well-posedness for $s > 1 + \frac{\nu}{2}$ \cite{AitzhanAmbrose2024}. For simplicity of exposition, we show a well-posedness result for small data on the unit time scale:
\begin{theorem}
\label{thm:QPBenjaminOno}
Let $s>\frac{\nu+1}{2}$. There is $c>0$ such that for any $u_0 \in H^s_{\Lambda}$ with $\| u_0 \|_{H^s_\Lambda} \leq c$ there is a unique solution $u \in F^s(T) \hookrightarrow C([0,1],H^s_{\Lambda})$ and continuous dependence of the solution on the initial data in $H^s_\Lambda$.
\end{theorem}


Already for decaying data in standard $L^2$-Sobolev spaces $H^s(\R)$ the Benjamin-Ono equation is notorious for being insoluble via the contraction mapping principle \cite{MolinetSautTzvetkov2001,KochTzvetkov2005}. The scaling critical regularity is $s_c = -\frac{1}{2}$. A breakthrough result was the employment of a gauge transformation reminiscent of the Cole-Hopf transform, which goes back to Tao \cite{Tao2004}. We refer to \cite{KochTzvetkov2003,IonescuKenig2007,BurqPlanchon2008,
Molinet2008} for further reading.

Finally, Killip--Laurens--Vi\c{s}an \cite{KillipLaurensVisan2024} proved the sharp well-posedness for regularities $s > -\frac{1}{2}$ on the real line by using the complete integrability. This was preceded by the sharp result on the torus due to G\'erard--Kappeler--Topalov \cite{GerardKappelerTopalov2023}. To the best of the author's knowledge, the only result for quasi-periodic data, which goes beyond energy arguments, is due to Papenburg \cite{Papenburg2025BenjaminOno} who combined a gauge transform adapted to quasi-periodic functions with the Strichartz estimates from \cite{Schippa2025Quasiperiodic}. In terms of regularity, Theorem \ref{thm:QPBenjaminOno} (modestly) improves on the local well-posedness result in \cite{Papenburg2025BenjaminOno}. The key aspect is that it does not depend on a gauge transform. This makes it possible to extend the argument to other dispersion relations or higher nonlinearities $\partial_x (u^k)$, $k \geq 3$; see \cite{ShorttimeFourierTransformRestriction}.




\smallskip

\emph{Outline of the paper.} In Section \ref{section:Prelim} we revisit Sobolev spaces for quasi-periodic functions and define the short-time adapted function spaces, in which we control the evolution. In Section \ref{section:QuasilinearLWP} we show how the estimates in short-time function spaces yield the quasilinear local well-posedness via frequency envelopes. The KdV equation \eqref{eq:KdVIntro} is analyzed in detail. In Section \ref{section:Trilinear} we prove multilinear estimates in adapted function spaces, which are applied in Section \ref{section:NonlinearEstimates} to prove short-time nonlinear and energy estimates for the KdV equation. In Section \ref{section:BenjaminOno} we show how a trilinear estimate obtained in Section \ref{section:Trilinear}, combined with the approach detailed in the context of the KdV equation in the previous sections, yields Theorem \ref{thm:QPBenjaminOno}.

\section{Preliminaries and function spaces}
\label{section:Prelim}

\subsection{Sobolev spaces for almost periodic functions}
We define the mean quadratic norm by
\begin{equation*}
\| u \|_{\mathcal{L}^2(\R)}^2 = \lim_{T \to \infty} \frac{1}{2T} \int_{-T}^T |u(x)|^2 dx < \infty.
\end{equation*}
The paper is concerned with the analysis of quasi-periodic functions. To ease notation, when denoting integrals in space, we do not indicate the spatial mean. We denote averaged multilinear expressions as
\begin{equation*}
\int_{\R} u_1 u_2 \ldots u_m d\bar{x} := \lim_{T \to \infty} \frac{1}{2T} \int_{-T}^T u_1 \ldots u_m dx.
\end{equation*}

For a trigonometric polynomial, i.e., a linear combination of complex exponentials, we find for any $M \in \N$, $(k_i)_{i \in \N}$ with $k_i \neq k_j$ for $i \neq j$,
\begin{equation*}
\big\| \sum_{i =1 }^M a_i e^{i k_i x} \big\|_{\mathcal{L}^2(\R)}^2 = \sum_{i = 1}^M |a_i|^2.
\end{equation*}
The set of $2$-Besicovitch almost periodic functions is given by the closure of trigonometric polynomials in the averaged $L^2$-norm:
\begin{equation*}
\mathcal{B}_2(\R) = \overline{\{ \sum_{\lambda \in \Lambda} a_\lambda e^{i k_\lambda x} : \, \Lambda \subseteq \R, \; \# \Lambda < \infty \}}^{\| \cdot \|_{\mathcal{L}^2}}.
\end{equation*}
It is well-known that for $f \in \mathcal{B}_2(\R)$ we have the (not necessarily pointwise convergent) representation in the Hilbert space sense
\begin{equation*}
f(x) = \sum_{\lambda \in \Lambda} \hat{f}(\lambda) e^{i \lambda x} \text{ with } \hat{f}(\lambda) = \frac{1}{2T} \int_{-T}^T f(x) e^{-i \lambda x} dx.
\end{equation*}
with $\Lambda$ at most countably infinite and 
\begin{equation*}
\| f \|^2_{\mathcal{B}_2(\R)} = \sum_{\lambda \in \Lambda} |\hat{f}(\lambda)|^2.
\end{equation*}
Let $\nu \geq 1$, $\omega \in \R^{\nu}_{>0}$ with $(\omega_1,\ldots,\omega_{\nu})$ being rationally independent.
We let $\sigma(f) = \{ \lambda \in \Lambda : \hat{f}(\lambda) \neq 0 \}$ denote the spectrum of $f$. Recall the following classification of almost periodic functions:
\begin{itemize}
\item If $\sigma(f) \subseteq \Lambda = \omega \Z$, then $f$ is a \emph{periodic function}.
\item If $\sigma(f) \subseteq \Lambda = \omega_1 \Z + \ldots + \omega_{\nu} \Z$ with $\omega \in \R^{\nu}_{>0}$ a rationally independent vector, then $f$ is a \emph{quasi-periodic function}.
\item If $\sigma(f) \subseteq \Lambda = \sum_{k=1}^{\infty} \omega_k \Z$ with $\omega \in \R^{\N}_{>0}$ a rationally independent vector, then $f$ is an \emph{almost-periodic function}. 
\end{itemize}

In the following we consider \emph{quasi-periodic initial data}. For $\nu \in \N$ let $\omega \in \R^{\nu}_{>0}$ be a rationally independent vector. We consider functions on the real line with spectrum in the lattice $\Lambda = \omega_1 \Z + \omega_2 \Z + \ldots + \omega_{\nu} \Z$, which are in the distributional sense given by
\begin{equation*}
	f(x) = \sum_{k \in \Z^{\nu}} e^{i k_{\omega} x} \hat{f}(k_{\omega}).
\end{equation*}
Above we let $k_{\omega} = k \cdot \omega$ and $\langle k \rangle = (1+|k|^2)^{\frac{1}{2}}$. We let $b=\nu-1$ - this is the dimension of frequencies \emph{transverse} to $\omega$. These frequencies are hard to control due to the lack of dispersion. The \emph{height} of a frequency $k_\omega$ is defined by $|k|$. We consider Sobolev spaces of quasi-periodic functions with regularity $s \geq 0$:
\begin{equation*}
	H^{s}_{\Lambda} = \{ f \in \mathcal{B}_2(\R) : \, \sigma(f) \subseteq \Lambda, \; \| f \|^2_{H^s} = \sum_{k \in \Z^{\nu}}  \langle k \rangle^{2s} |\hat{f}(k_{\omega} )|^2 < \infty \}.
\end{equation*}

\begin{remark}
	A more natural choice of norms might be to consider instead weights in the tangential frequencies $k_{\omega}$:
	\begin{equation*}
		\| f \|_{H^{s_1,s_2}_{\Lambda}}^2 = \sum_{k \in \Z^{\nu}} \langle k \rangle^{2 s_1} \langle k_\omega \rangle^{2 s_2} |\hat{f}(k_\omega)|^2.
	\end{equation*}
	By this we indeed find a refinement of the short-time nonlinear estimate, but it causes a problem in commutator estimates which are required to prove energy estimates; cf. the proof of Proposition \ref{prop:EnergyEstimates}.
\end{remark}

\subsection{Frequency localization and adapted $U^p$-$V^p$-spaces}

We define several frequency localization operators.
Let $N,M \in 2^{\Z}$. Let $\chi_1 \in C^\infty_c(B_1(0,2))$ be a radially decreasing function with $\chi_1 \equiv 1$ on $B_1(0,1)$. For $N \in 2^{\N} = \{2,4,\ldots \}$ we let $\chi_N(\xi) = \chi_1(\xi / N) - \chi_1(2 \xi/N)$. Clearly, $1 \equiv \sum_{N \in 2^{\N_0}} \chi_N(x)$.

 We consider the frequency-localization $P_N$, which localizes frequencies $|k_{\omega}| \sim N$:
\begin{equation*}
	P_N u (x) = \sum_{k \in \Z^{\nu}} \chi_N(k_{\omega}) e^{i k_{\omega} x } \hat{u}(k_{\omega}).
\end{equation*} 
We consider the height localization $R_M$, which localizes frequencies $k_\omega$ to $\langle k \rangle \sim M$:
\begin{equation*}
	R_M u(x) = \sum_{k \in \Z^{\nu}} \chi_M(\langle k \rangle) e^{i k_{\omega} x} \hat{u}(k_{\omega}).
\end{equation*}

For an introduction to Banach-valued spaces of bounded variation $V^p H^r_{\Lambda}$, and its predual $U^p H^r_{\Lambda}$ in the present context we refer to \cite[Sections~2.3--2.5]{ShorttimeFourierTransformRestriction}.
The adapted function spaces $V^p_{Ai} H^r_{\Lambda} = e^{t \partial_x^3} V^p H^r_{\Lambda}$, $U^p_{Ai} H^r_{\Lambda} = e^{t \partial_x^3} U^p H^r_{\Lambda}$ are defined like in \cite[Section~2]{ShorttimeFourierTransformRestriction}. 
We define for an interval $I \subseteq \R$
\begin{equation*}
DU^2(I) = \{ \partial_t u : u \in U^2(I) \}
\end{equation*}
with the derivative taken in the sense of tempered distributions. It holds $L^1(I) \hookrightarrow DU^2(I)$ by \cite[Lemma~2.3.6]{ShorttimeFourierTransformRestriction}. 
In \cite[Lemma~2.3.7]{ShorttimeFourierTransformRestriction} is pointed out 
\begin{equation*}
\| f \|_{DU^2(I)} = \sup_{v \in V^2_0(I) = 1} \big| \int_I \int f \bar{v} d \bar{x} dt \big|.
\end{equation*}

\smallskip

Let
\begin{equation*}
I_N =
\begin{cases}
[0,2], \, &N =1, \\
[N/2,2N], \, &N \in 2^{\N}.
\end{cases}
\end{equation*}
We consider short-time adapted function spaces for frequency-localized functions for $N \in 2^{\N_0}$:
\begin{equation*}
\begin{split}
F_{N,r}(T) &= \{ u \in U^2_{Ai} H^0_{\Lambda} : \sigma( u) \subseteq I_N,  \\ 
&\quad \; \| u \|_{F_{N,r}(T)} = \sup_{\substack{I \text{ an interval } \subseteq [0,T]: \\ |I| = \min( N^{-\frac{3}{2}}, T)}} \| 1_{I}(t) u \|_{U^2_{Ai} H^r_{\Lambda}} < \infty \}.
\end{split}
\end{equation*}
Above $1_I$ denotes the indicator function of $I$. When $r=0$, the subindex will be omitted to lighten notations.
The frequency-dependent time localization hinges only on the tangential frequencies. We assemble the short-time function space by Littlewood-Paley decomposition:
\begin{equation*}
\| u \|^2_{F^{r}(T)} = \sum_{N \geq 1} \| P_N u \|^2_{F_{N,r}(T)}.
\end{equation*}
This will be the function space, in which we propagate the solution.
The frequency-dependent time localization $T=T(N)= N^{-\frac{3}{2}}$ does not reduce the periodic geometry to the Euclidean geometry (see \cite[Chapter~3]{ShorttimeFourierTransformRestriction}). Rather, it simultaneouly allows us to overcome the derivative loss in the nonlinearity, captures the smallness of time intervals $(0,T)$, $T \in (0,1)$ (technically, this amounts to leeway in the modulation regularity), and yields favorable energy estimates as well. Indeed, any choice $T=T(N)=N^{-\alpha}$, $\alpha \in (1,2)$ works for the KdV equation.

The dual space to capture the nonlinearity is defined by
\begin{equation*}
\begin{split}
\mathcal{N}_{N,r}(T) &= \{ u \in DU^2_{Ai} H^0_\Lambda : \sigma(u) \subseteq I_N, 
 \\ &\quad \; \| u \|_{\mathcal{N}_N(T)} = \sup_{\substack{I \text{ an interval } \subseteq [0,T] : \\  | I| = \min(T,  N^{-\frac{3}{2}})}} \| 1_I(t) u \|_{DU^2_{Ai} H^r_{\Lambda}} < \infty \}.
\end{split}
\end{equation*}
The space $\mathcal{N}^{r}(T)$ is assembled by Littlewood-Paley theory like above. We define an energy norm, which considers frequency-localized contributions separately:
\begin{equation*}
\| u \|_{E^{r}(T)}^2 = \sum_{N \geq 1} \sup_{t \in [0,T]} \| P_N u(t) \|_{H^r_{\Lambda}}^2.
\end{equation*}
We consider the function space:
\begin{equation*}
E^{r}(T) = \{ u \in C([0,T];H^0_{\Lambda}) \, : \, \| u \|_{E^{r}(T)} < \infty \}.
\end{equation*}

Analogous to the periodic case we have the short-time linear energy estimate due to the Duhamel principle (cf. \cite[Lemma~2.4.1]{ShorttimeFourierTransformRestriction}):
\begin{lemma}[Short-time linear energy estimate]
\label{lem:LinearEnergyEstimate}
Let $r \geq 0$, and $u \in F^{r}(T)$, $v \in L^1([0,T],H^r_\Lambda)$ be a Duhamel solution to 
\begin{equation*}
\partial_t u + \partial_x^3 u = v.
\end{equation*} 
Then it holds
\begin{equation*}
\| u \|_{F^{r}(T)} \lesssim \| u \|_{E^{r}(T)} + \| v \|_{\mathcal{N}^{r}(T)}.
\end{equation*}
\end{lemma}

\section{Proof of the quasilinear local well-posedness result}
\label{section:QuasilinearLWP}

This section gives an overview of the argument. The proof of local well-posedness for quasilinear equations via short-time Fourier restriction is standard, so we shall be brief. It follows the general principles:
\begin{enumerate}
\item The proof of a priori estimates for solutions,
\item The proof of estimates for differences of solutions in a weaker topology,
\item The conclusion of continuous dependence of solutions in the original topology via frequency envelopes.
\end{enumerate}
Frequency envelopes were already used by Tao \cite{Tao2004} to prove global well-posedness of the Benjamin-Ono equation on the real line in $H^1(\R)$.

\subsection{Existence of regular solutions}

The starting point of our proof of low regularity well-posedness is the existence of strong solutions to \eqref{eq:KdVIntro} with $u_0 \in H^r_{\Lambda}$, $r = 10 \nu$, which are denoted by $u = S_T^{\infty}(u_0) \in C([0,T],H^r_\Lambda)$ with $T=T(\| u_0 \|_{H^r_\Lambda})$. This was proved via energy arguments for the Benjamin-Ono equation for $r > \frac{\nu}{2}+1$ in \cite{AitzhanAmbrose2024} and extends to the KdV equation. Then we shall see how the short-time analysis allows us to control the solutions depending on much lower regularity. The data-to-solution mapping at low regularities is obtained as extension of the mapping at higher regularities. This requires in the first step to argue that the solutions exist on times depending on the low regularity.
A solution $u \in C([0,T],H^r_{\Lambda})$ to \eqref{eq:KdVIntro} is referred to as \emph{strong} provided that
\begin{equation*}
u(t) = e^{t \partial_x^3} u_0 + \int_0^t e^{(t-s)\partial_x^3} (u \partial_x u)(s) ds
\end{equation*}
and $u \partial_x u \in \mathcal{L}^2$. For $r > \frac{\nu+1}{2}$, the latter condition is satisfied by the algebra property of Sobolev spaces.

We remark that the short-time analysis can be modified to construct solutions via Galerkin approximation, considering the truncated non-linearity
\begin{equation*}
\left\{ \begin{array}{cl}
\partial_t u + \partial_x^3 u &= R_M (R_M u \partial_x  R_M u), \\
u(0) &= u_0 \in H^s_\Lambda.
\end{array} \right.
\end{equation*}
This nonlinearity is bounded in $H^s_\Lambda$ for $s \geq 0$, which yields the existence of global solutions in $H^0_\Lambda$. Proving short-time bounds independently of $M$ yields solutions by compactness arguments. For a different model this was detailed in \cite{HerrSchippaTzvetkov2026}, to which we refer for the conclusion of local well-posedness via short-time estimates. Precisely, we prove the following theorem:
\begin{theorem}
For $s > \frac{\nu+1}{2}$ the data-to-solution mapping $S_T^{\infty}$ extends uniquely to a continuous flow. For any $R>0$ there is a $T_R > 0$ and a continuous flow $S_{T_R}^s : B_{H^s_\Lambda}(0,R) \to C([0,T_R],H^s_{\Lambda})$, which assigns to any $u_0 \in B_{H^s_\Lambda}(0,R)$ a strong solution $u \in F^s(T) \hookrightarrow C([0,T],H^s_\Lambda)$. The dependence $T=T_R > 0$ is lower semicontinuous and $T \gtrsim 1$ as $\| u_0 \|_{H^s_\Lambda} \to 0$. 
\end{theorem}

\subsection{A priori estimates for solutions}

For solutions $u$ to the KdV equation with regular quasi-periodic initial data $u_0 \in H^{10 r'}_\Lambda$ we shall establish the following set of estimates for $r' \geq r > \frac{b}{2}+1$:
\begin{equation*}
\left\{ \begin{array}{cl}
\| u \|_{F^{r'}(T)} &\lesssim \| u \|_{E^{r'}(T)} + \| \partial_x (u^2) \|_{\mathcal{N}^{r'}(T)}, \\
\| \partial_x (u^2) \|_{\mathcal{N}^{r'}(T)} &\lesssim T^{\frac{1}{4}} \| u \|_{F^{r}(T)} \| u \|_{F^{r'}(T)}, \\
\| u \|_{E^{r'}(T)}^2 &\lesssim \| u_0 \|^2_{H^{r'}_{\Lambda}} + T^{\frac{1}{2}} \| u \|_{F^{r'}(T)}^2 \| u \|_{F^{r}(T)}
\end{array} \right.
\end{equation*}
for $0<T\leq 1$. The above set of estimates is the content of Lemma \ref{lem:LinearEnergyEstimate}, Propositions \ref{prop:NonlinearEstimate} and \ref{prop:EnergyEstimates}. Note that for the regular solutions at hand, the time of existence $T^*$ depends on $\| u_0 \|_{H^{10 r'}_\Lambda}$.

To find a priori estimates, we apply the above for $r'=r$.
The estimate
\begin{equation*}
\| u \|_{F^{r}(T)}^2 \lesssim \| u_0 \|^2_{H^{r}_\Lambda} + T^{\frac{1}{4}} (\| u \|_{F^{r}(T)}^3 + \| u \|^4_{F^{r}(T)})
\end{equation*}
is immediate. By the estimates (cf. \cite{IonescuKenigTataru2008})
\begin{equation*}
\limsup_{T \to 0} \| u \|_{E^{r}(T)} \lesssim \| u_0 \|_{H^{r}_{\Lambda}}, \quad \limsup_{T \to 0} \| u \partial_x u \|_{\mathcal{N}^{r}(T)} = 0,
\end{equation*}
it follows from a standard continuity argument that
\begin{equation*}
\| u \|_{F^{r}(T)} \lesssim \| u_0 \|_{H^{r}_{\Lambda}}
\end{equation*}
for $T=T(\| u_0 \|_{\mathcal{H}^{r}_{\Lambda}}) \leq T^*$. Using persistence of regularity, we find that $T=T(\| u_0 \|_{H^r_\Lambda})$ with $T \gtrsim 1$ as $u_0 \to 0$.

\subsection{A priori estimates for differences of solutions}

With a priori estimates for solutions at hand, we can turn to estimating differences of solutions $v = u_1 - u_2$ on times $T=T(\| u_1(0) \|_{H^r_\Lambda}, \| u_2(0) \|_{H^r_\Lambda})$ which satisfy the equation:
\begin{equation*}
\partial_t v + \partial_x^3 v = \partial_x (v(u_1 + u_2))/2.
\end{equation*}
We shall prove the following set of estimates for $v$ on the level of $\mathcal{L}^2$:
\begin{equation*}
\left\{ \begin{array}{cl}
\| v \|_{F^{0}(T)} &\lesssim \| v \|_{E^{0}(T)} + \| \partial_x (v(u_1+u_2))\|_{\mathcal{N}^{0}(T)}, \\
\| \partial_x (v(u_1+u_2))\|_{\mathcal{N}^{0}(T)} &\lesssim T^{\frac{1}{4}} \| v \|_{F^{0}(T)} (\| u_1 \|_{F^{r}(T)} + \| u_2 \|_{F^{r}(T)}), \\
\| v \|^2_{E^{0}(T)} &\lesssim \| v(0) \|^2_{H_\Lambda^{0}(T)} + T^{\frac{1}{2}} \| v \|_{F^{0}(T)}^2 (\| u_1 \|_{F^{r}(T)} + \| u_2 \|_{F^{r}(T)}).
\end{array} \right.
\end{equation*}
This is the content of Lemma \ref{lem:LinearEnergyEstimate}, Propositions \ref{prop:NonlinearEstimate} and \ref{prop:EnergyEstimateDifferences}.
By the a priori estimates for solutions we obtain for $T=T(\| u_i(0) \|_{H_{\Lambda}^{r}})$ Lipschitz-continuous dependence of differences in $H^0_{\Lambda}$ depending on higher regularity:
\begin{equation*}
\| v \|_{F^{0}(T)} \lesssim \| v(0) \|_{\mathcal{L}^2_x}.
\end{equation*}

\subsection{Conclusion of local well-posedness}

With the a priori estimates in the original topology and the Lipschitz dependence of solutions in a weaker topology at hand, we can conclude the continuous dependence of solutions in the $H^r_{\Lambda}$-norm via frequency envelopes \cite{Tao2004}. Since this has become folklore, the details are omitted.
$\hfill \Box$

\section{Multilinear estimates in adapted function spaces}
\label{section:Trilinear}
In this section we prove bi- and trilinear estimates in function spaces adapted to dispersive equations. Together with the frequency-dependent time localization, these will allow us to control the derivative nonlinearity.
We begin with a bilinear estimate, which relies on the bilinear C\'ordoba--Fefferman square function estimate. First, we recall the linear C\'ordoba--Fefferman square function estimate. Let $\pm \mathbb{P}^1 = \{ (\xi,\pm \xi^2) : 0 < \xi < 2 \}$ denote the graph of the positive resp. negative parabola, and $\mathcal{N}_{\delta}(\pm \mathbb{P}^1)$ denote the $\delta$-neighborhood of $\pm \mathbb{P}^1$. For $0< \alpha \leq 1$ we denote by $\mathbb{I}_{\alpha}$ a finitely overlapping partition of $(-3,3)$ into $\alpha$-intervals.

For $F \in \mathcal{S}(\R^2)$, $\theta \in \mathbb{I}_{\alpha}$ we denote by $F_{\theta}$ the Fourier projection  of the first frequency coordinate to $\theta$ given by $\hat{F}_{\theta}(\xi) = 1_{\theta}(\xi_1) \hat{F}(\xi_1,\xi_2)$. By $\pi_1 : \R^2 \to \R$ the projection $(x_1,x_2) \mapsto x_1$ is deonted. We have the following square function estimate:
\begin{theorem}[C\'ordoba--Fefferman square function estimate]
Let $0<\delta \leq 1$, $F \in \mathcal{S}(\R^2)$ with $\text{supp}(\hat{F}) \subseteq \mathcal{N}_{\delta}(\pm \mathbb{P}^1)$. Then the following holds:
\begin{equation*}
\| F \|_{L^4(\R^2)} \lesssim \big\| \big( \sum_{\theta \in \mathbb{I}_{\delta^{1/2}}} |F_{\theta}|^2 \big)^{\frac{1}{2}} \big\|_{L^4(\R^2)}.
\end{equation*}
\end{theorem}
The square function estimate depends on the non-degeneracy of the curve and versions for $\Gamma_3 = \{ (\xi,\xi^3) : \xi \in (-2,2) \}$ need to respect the flatness at the origin. In \cite[Theorem~2.2]{Schippa2025Quasiperiodic} a square function estimate with $\delta^{1/3}$-intervals is proved.

The square function estimate can be used to estimate exponential sums by approximating the exponential sum with an oscillatory integral, carrying out the square function estimate and passing back to an exponential sum. This will be referred to as continuous approximation in the following. To the best of the author's knowledge, this approximation was first carried out by Bourgain \cite{Bourgain2013}.
In \cite[Section~3.1]{Schippa2025Quasiperiodic} was pointed out how the square function estimate together with the continuous approximation yields sharp estimates for quasi-periodic functions. 

\smallskip

Here we point out how the bilinear C\'ordoba--Fefferman square function estimate yields new bilinear estimates for exponential sums related to quasi-periodic functions. Under an additional transversality assumption we have the following refinement of the linear estimate:
\begin{theorem}[Bilinear C\'ordoba--Fefferman square function estimate for the cubic]
Let $0 < \delta \leq 1$, $F_i \in \mathcal{S}(\R^2)$, $i=1,2$ with $\text{supp}(\hat{F}_i) \subseteq \mathcal{N}_{\delta}(\Gamma_3)$ such that
\begin{equation*}
\text{dist}(\text{supp}(\pi_1 \hat{F}_1),\text{supp}(\pi_1 \hat{F}_2)) \geq 1/16 \text{ and }
\text{dist}(\text{supp}(\pi_1 \hat{F}_1),\text{supp}(- \pi_1 \hat{F}_2)) \geq 1/16.
\end{equation*}
Then it holds:
\begin{equation*}
\| F_1 F_2 \|_{L^2(\R^2)} \lesssim \big\| \big( \sum_{\theta_1 \in \mathbb{I}_{\delta}} |F_{\theta_1}|^2 \big)^{\frac{1}{2}} \big( \sum_{\theta_2 \in \mathbb{I}_{\delta}} |F_{\theta_2}|^2 \big)^{\frac{1}{2}} \big\|_{L^2(\R^2)}.
\end{equation*}
\end{theorem}
\begin{proof}
We can suppose by Cauchy-Schwarz that $\delta < 1/32$.
We let $F_i = \sum_{\theta \in \mathbb{I}_{\delta}} F_{i \theta}$ and obtain
\begin{equation*}
\big\| \sum_{\theta_1 \in \mathbb{I}_{\delta}} F_{1 \theta_1} \sum_{\theta_2 \in \mathbb{I}_{\delta}} F_{2 \theta_2} \big\|^2_{L^2(\R^2)} = \sum_{\substack{\theta_1,\theta_2, \\ \theta_3,\theta_4}} \int F_{1 \theta_1} F_{2 \theta_2} \overline{F}_{1 \theta_3} \overline{F}_{2 \theta_4} dx.
\end{equation*}
We find by Plancherel's theorem that
\begin{equation}
\label{eq:TrivialBiorthogonality}
\int F_{1 \theta_1} F_{2 \theta_2} \overline{F}_{1 \theta_3} \overline{F}_{2 \theta_4} dx = 0,
\end{equation}
unless there are $\xi_i \in 2 \theta_i$, which satisfy
\begin{equation*}
\left\{ \begin{array}{cl}
\xi_1 + \xi_2 &= \xi_3 + \xi_4, \\
\xi_1^3 + \xi_2^3 &= \xi_3^3 + \xi_4^3 + \mathcal{O}(\delta).
\end{array} \right.
\end{equation*}
Taking the third power of the first equation and subtracting the second, we find
\begin{equation*}
\left\{ \begin{array}{cl}
\xi_1 + \xi_2 &= \xi_3 + \xi_4, \\
\xi_1 \xi_2 (\xi_1+\xi_2) &= \xi_3 \xi_4(\xi_3 + \xi_4) + \mathcal{O}(\delta).
\end{array} \right.
\end{equation*}
By the separation of $\pi_1 \text{supp}(\hat{F}_1)$ and $-\pi_1 \text{supp}(\hat{F}_2)$, we obtain
\begin{equation*}
\left\{ \begin{array}{cl}
\xi_1 + \xi_2 &= \xi_3 + \xi_4, \\
2 \xi_1 \xi_2  &=  2 \xi_3 \xi_4 + \mathcal{O}(\delta).
\end{array} \right.
\end{equation*}
This is equivalent to
\begin{equation}
\label{eq:BiorthogonalityAux}
\left\{ \begin{array}{cl}
\xi_1 + \xi_2 &= \xi_3 + \xi_4, \\
\xi_1^2 + \xi_2^2  &=  \xi_3^2 + \xi_4^2 + \mathcal{O}(\delta)
\end{array} \right.
\end{equation}
and by the classical bilinear C\'ordoba--Fefferman argument, which depends on \\ $\text{dist}(\pi_1 \text{supp}(\hat{F}_1),\pi_1 \text{supp}(\hat{F}_2)) \geq 1/16$, we find that \eqref{eq:TrivialBiorthogonality} holds unless
\begin{equation}
\label{eq:Biorthogonality}
|\xi_1 - \xi_3| \lesssim \delta, \quad |\xi_2 - \xi_4 | \lesssim \delta,
\end{equation}
from which the claim is immediate by Cauchy-Schwarz. 

For self-containedness we note that \eqref{eq:BiorthogonalityAux} yields indeed
\begin{equation*}
\left\{ \begin{array}{cl}
\xi_1 + \xi_2 &= \xi_3 + \xi_4, \\
(\xi_1- \xi_4)(\xi_1+\xi_4)  &=  (\xi_3-\xi_2)(\xi_3+\xi_2) + \mathcal{O}(\delta)
\end{array} \right.
\end{equation*}
and furthermore, by the separation condition,
\begin{equation*}
\left\{ \begin{array}{cl}
\xi_1 + \xi_2 &= \xi_3 + \xi_4, \\
\xi_1+\xi_4  &=  \xi_3+\xi_2 + \mathcal{O}(\delta),
\end{array} \right.
\end{equation*}
which yields again \eqref{eq:Biorthogonality}.
\end{proof}

The argument is flexible and foremostly hinges on adequate separation conditions of the frequencies. We record the following variant, which plays the key role when estimating solutions to the Benjamin-Ono equation.

\begin{theorem}[C\'ordoba--Fefferman bilinear square function estimate, II]
\label{thm:SFBO}
Let $0<\delta\leq 1$, $F_i \in \mathcal{S}(\R^2)$, $i=1,2$ with $\text{supp}(\hat{F}_i) \subseteq \mathcal{N}_{\delta}(\mathbb{P}^1)$ or $\text{supp}(\hat{F}_i) \subseteq \mathcal{N}_{\delta}(-\mathbb{P}^1)$ and 
\begin{equation*}
\text{dist}(\text{supp}(\pi_1 \hat{F}_1),\text{supp}(\pi_1 \hat{F}_2)) \geq 1/16 \text{ and }
\text{dist}(\text{supp}(\pi_1 \hat{F}_1),\text{supp}(- \pi_1 \hat{F}_2)) \geq 1/16.
\end{equation*}
Then it holds:
\begin{equation*}
\| F_1 F_2 \|_{L^2(\R^2)} \lesssim \big\| \big( \sum_{\theta_1 \in \mathbb{I}_{\delta}} |F_{\theta_1}|^2 \big)^{\frac{1}{2}} \big( \sum_{\theta_2 \in \mathbb{I}_{\delta}} |F_{\theta_2}|^2 \big)^{\frac{1}{2}} \big\|_{L^2(\R^2)}.
\end{equation*}
\end{theorem}
\begin{proof}
By the same initial steps like above, we find
\begin{equation*}
\int F_{1 \theta_1} F_{2 \theta_2} \overline{F}_{1 \theta_3} \overline{F}_{2 \theta_4} = 0
\end{equation*}
unless for $\xi_i \in \theta_i$ it holds
\begin{equation*}
\left\{ \begin{array}{cl}
\xi_1 + \xi_2 &= \xi_3 + \xi_4, \\
a_1 \xi_1^2 + a_2 \xi_2^2 &= a_1 \xi_3^2 + a_2 \xi_4^2 + \mathcal{O}(\delta)
\end{array} \right.
\end{equation*}
with
\begin{equation*}
a_i = \begin{cases}
1, &\quad \text{supp}(\hat{F}_i) \subseteq \mathcal{N}_{\delta}(\mathbb{P}^1), \\
-1, &\quad \text{supp}(\hat{F}_i) \subseteq \mathcal{N}_{\delta}(-\mathbb{P}^1).
\end{cases}
\end{equation*}
By symmetry it suffices to consider $(a_1,a_2) = (1,1)$ and $(a_1,a_2) = (1,-1)$. The first case was covered above.
When $(a_1,a_2)=(1,-1)$, rewriting the second equation and the separation of the Fourier support projected to the first coordinate gives
\begin{equation*}
(\xi_1-\xi_2)(\xi_1+\xi_2) = (\xi_3-\xi_4)(\xi_3+\xi_4) + \mathcal{O}(\delta) \Rightarrow \xi_1 - \xi_2 = \xi_3 - \xi_4 + \mathcal{O}(\delta).
\end{equation*}
Hence, $\text{dist}(\theta_1,\theta_3) + \text{dist}(\theta_2,\theta_4) \leq C \delta$ from which the claim is immediate.
\end{proof}

In \cite{Schippa2025Trilinear} was pointed out how the bilinear square function estimate, together with the continuous approximation, yields the following bilinear short-time Strichartz estimate for Schr\"odinger evolutions on tori for separated frequencies $N_1 \geq 8 N_2$:
\begin{equation*}
\| e^{it \partial_x^2} P_{N_1} u_1 e^{it \partial_x^2} P_{N_2} u_2 \|_{L^2_{t,x}([0,N_1^{-1}] \times \T)} \lesssim N_1^{-\frac{1}{2}} \| P_{N_1} u_1 \|_{L^2(\T)} \| P_{N_2} u_2 \|_{L^2(\T)}.
\end{equation*}
The following short-time bilinear estimate for quasi-periodic functions is the backbone of the low regularity well-posedness result:
\begin{proposition}
\label{prop:ShorttimeBilinearEstimateAiry}
Let $K \leq 2N$ and $I_j \subseteq (-3 N,3 N)$, $j=1,2$ intervals of length $K$. Let $\tilde{I}_j = \{ x \in \R_{>0} : x \in I_j \vee - x \in I_j \}$ and assume that $\text{dist}(\tilde{I}_1,\tilde{I}_2) \geq N/16$. Assume $u_1,u_2 \in H^0_\Lambda$. Let $N^{-2} \leq T \leq N^{-1}$. The following bilinear estimate holds:
\begin{equation}
\label{eq:BilinearStrichartzKdV}
\| P_{I_1} R_{M_1} e^{t \partial_x^3} u_1 P_{I_2} R_{M_2} e^{t \partial_x^3} u_2 \|_{L^2_t([0,T],\mathcal{L}^2_x)} \lesssim T^{\frac{1}{2}} M_{\min}^{\frac{\nu-1}{2}} \| u_1 \|_{H^0_\Lambda} \| u_2 \|_{H^0_\Lambda}.
\end{equation}
\end{proposition}
\begin{proof}
In the first step, we carry out an almost orthogonal decomposition in the height - to localize $u_i$ to frequency cubes $R_x$ in the height which are of minimal size.
Suppose by symmetry that $M_2 = M_{\min}$. We obtain
\begin{equation}
\label{eq:AlmostOrthogonalAiry}
\begin{split}
&\; \| P_{I_1} R_{M_1} e^{t \partial_x^3} u_1 P_{I_2} R_{M_2} e^{t \partial_x^3} u_2 \|_{L_t^2([0,N^{-1}],\mathcal{L}^2_x)} \\
&\lesssim \big( \sum_{x \in 2M_2 \Z^d} \| P_{I_1} R^x_{M_2} e^{t \partial_x^3} u_1 P_{I_2} R_{M_2} e^{t \partial_x^3} u_2 \|^2_{L_t^2([0,N^{-1}];\mathcal{L}^2_x)} \big)^{\frac{1}{2}}
\end{split}
\end{equation}
with $R_{M_2}^x$ denoting a Fourier projection of the frequency height to a cube of sidelength $M_2$ centered at $x \in M_2 \Z^d$.
By this it suffices to prove the estimate for frequency height projections to cubes of length $M_1 = M_2$. We still use the notation from \eqref{eq:BilinearStrichartzKdV} to lighten the notation.

Next, we want to apply the bilinear C\'ordoba--Fefferman square function estimate. To this end, we carry out a rescaling $t \to N^3 t$, $x \to N x$, $\xi \to \xi / N$, which restricts the frequencies to $(-3,3)$.
\begin{equation*}
\begin{split}
&\quad \| P_{I_1} R_{M_1} e^{t \partial_x^3} u_1 P_{I_2} R_{M_2} e^{t \partial_x^3} u_2 \|_{L_t^2([0,T], \mathcal{L}^2_x)} \\
&= \| P_{I'_1} R_{M_1'} e^{t \partial_x^3} u'_1 P_{I'_2} R_{M'_2} e^{t \partial_x^3} u'_2 \|_{L_t^2([0,T N^3], \mathcal{L}^2_x)}
\end{split}	
\end{equation*}
with the rescaled quantities indicated by $'$. The Jacobian is absorbed into the rescaled functions as we will later reverse the scaling. By this we dominate
\begin{equation*}
\begin{split}
&\quad \| P_{I'_1} R_{M_1'} e^{t \partial_x^3} u'_1 P_{I'_2} R_{M'_2} e^{t \partial_x^3} u'_2 \|^2_{L_t^2([0,TN^3], \mathcal{L}^2_x)} \\
&= \lim_{L \to \infty} \frac{1}{2NL} \int_{[0,TN^3] \times [-NL,NL]} \big| e^{t \partial_x^3} P_{I_1'} R_{M_1'}  u_1' e^{t \partial_x^3} P_{I_2'} R_{M_2'} u_2 \big|^2 dx dt.
\end{split}
\end{equation*}
Choose $L \geq 16N$, and let $w_A: \R \to \R$ be a function with compact Fourier support, which satisfies
$|w_A(x)| \gtrsim 1$ for $|x| \lesssim A$ and $\text{supp}(\hat{w}_A) \subseteq B(0,c A^{-1})$. We can dominate
\begin{equation*}
\begin{split}
&\; \int_{[0,TN^3] \times [-NL,NL]} \big| e^{t \partial_x^3} P_{I_1'} R_{M_2'}^{x'}  u_1' e^{t \partial_x^3} P_{I_2'} R_{M_2'} u_2 \big|^2 dx dt \\ &\lesssim \int_{\R^2} w^4_{T N^3}(t) w_{NL}^4(x) \big| \ldots \big|^2 dx dt = \int_{\R^2} |F_1 F_2|^2 dx dt,
\end{split}
\end{equation*}
letting $F_i = w_{TN^3}(t) w_{NL}(t) e^{t \partial_x^3} P_{I_i'} R_{M_{i}'}  u_i'$. 
The resulting expressions have Fourier support in $\mathcal{N}_{\delta}(\Gamma_3)$ and are amenable to the bilinear square function estimate based on $\text{dist}(\pi_1(\text{supp}(\hat{F}_1)),\pm \pi_1(\text{supp}(\hat{F}_2))) \geq \frac{1}{16}$. Consequently,
\begin{equation*}
\int_{\R^2} \big| F_1 F_2 \big|^2 dx dt \lesssim \int_{\R^2} \big( \sum_{\theta'_1 \in \mathbb{I}_{(TN^3)^{-1}}} |F_{1 \theta'_1}|^2 \big) \big( \sum_{\theta'_2 \in \mathbb{I}_{(TN^3)^{-1}}} | F_{2 \theta'_2}  |^2 \big). 
\end{equation*}
By the same limiting argument as in \cite{Schippa2025Quasiperiodic} and reversing the scaling, we find
\begin{equation*}
\begin{split}
&\quad \| P_{I_1} R_{M_1} e^{t \partial_x^3} u_1 P_{I_2} R_{M_2} e^{t \partial_x^3} u_2 \|_{L^2_t([0,T],\mathcal{L}^2_x)} \\
&\lesssim \big\| \big( \sum_{\theta_1 \in \overline{\mathbb{I}}_1} |P_{I_1} R_{M_1} e^{t \partial_x^3} P_{\theta_1} u_1 |^2 \big)^{1/2} \big( \sum_{\theta_2 \in \overline{\mathbb{I}}_2} | P_{I_2} R_{M_2} e^{t \partial_x^3} P_{\theta_2} u_1 |^2 \big)^{1/2} \big\|_{L^2_t(w_{T}(t),\mathcal{L}^2)}
\end{split}
\end{equation*}
with $\overline{\mathbb{I}}_1 = \{ T^{-1} N^{-2} (k,k+1] : k \in \Z \}$ and $P_{\theta_i}$ denotes the corresponding frequency projection.
Applying Minkowski's inequality and H\"older's inequality we find
\begin{equation*}
\begin{split}
&\quad \| \big( \sum_{\theta_1 \in \overline{\mathbb{I}}_1} |P_{I_1} R_{M_1} e^{t \partial_x^3} P_{\theta_1} u_{1 } |^2 \big)^{\frac{1}{2}} \big( \sum_{\theta_2 \in \overline{\mathbb{I}}_2} |P_{I_2} R_{M_2} e^{t \partial_x^3} P_{\theta_2} u_{2} |^2 \big)^{\frac{1}{2}} \big\|_{L^2_t(w_{T}(t), \mathcal{L}^2_x)} \\
&\lesssim \prod_{i=1}^2 \big( \sum_{\theta_i \in \overline{\mathbb{I}}_i} \| P_{I_i} R_{M_i} e^{t \partial_x^3} P_{\theta_i} u_{i } \|^2_{L_t^4(w_{T}(t);\mathcal{L}^4_x)} \big)^{\frac{1}{2}} 
\end{split}
\end{equation*}
with $w_{T}$ a function which satisfies $w_{T}(t) \leq C_m (1+T^{-1} |t|)^{-m}$ for any $m \geq 0$.
We count the frequencies $k_{\omega}$ at fixed height which are accummulating in unit intervals like in \cite{Schippa2025Quasiperiodic}. Once $b= \nu-1$ components are prescribed, the remaining one is determined up to finitely many values (depending on $\omega$), and we have the simple bound $\# \{ k \in \Z^{\nu} : k_\omega \in \theta, \, k \in R_{M} \} \lesssim M^{\nu-1} = M^b$ with $\theta \in \mathbb{I}_1$ and $R_M$ denoting a cube of length $M$, like in \cite{Schippa2025Quasiperiodic}. Applying Bernstein's inequality (cf. \cite{Schippa2025Quasiperiodic}) yields
\begin{equation*}
\| P_{N_i} R_{M_i} e^{t \partial_x^3} P_{\theta_i} u_{i} \|_{\mathcal{L}^4_x} \lesssim M_i^{\frac{b}{4}} \| P_{N_i} R_{M_i} P_{\theta_i} u_i \|_{H^0_\Lambda}.
\end{equation*}
(Trivially) integrating in time we find
\begin{equation*}
\| P_{I_i} R_{M_i} e^{t \partial_x^3} P_{\theta_i} u_{i } \|_{L_t^4(w_{T}(t);\mathcal{L}^4_x)} \lesssim T^{\frac{1}{4}} M_i^{\frac{b}{4}} \| P_{I_i} R_{M_i} P_{\theta_i} u_{i } \|_{H^0_\Lambda}.
\end{equation*}
The almost orthogonal decompositions in $\theta_i$ can be summed up without loss in $\mathcal{L}^2_x$, and the claim follows.
\end{proof}

We have the following corollary regarding High-Low-interactions, which is immediate from an additional almost orthogonal decomposition in frequency intervals:
\begin{corollary}
\label{cor:BilinearHighLow}
Let $N_2 \leq N_1/8$. With notation like above, the following estimate holds:
\begin{equation*}
\| P_{N_1} R_{M_1} e^{t \partial_x^3} u_1 P_{N_2} R_{M_2} e^{t \partial_x^3} u_2 \|_{L_t^2([0,T],\mathcal{L}^2_x)} \lesssim T^{\frac{1}{2}} M_{\min}^{\frac{b}{2}} \prod_{i=1}^2 \| P_{N_i} R_{M_i} u_i \|_{H^0_\Lambda}.
\end{equation*}
\end{corollary}

We record the following trilinear estimate in adapted function spaces:
\begin{proposition}
\label{prop:TrilinearEstimate}
Let $N_i, M_i$, $i=1,2,3$ and $u_i \in V^2_{Ai} H^0_{\Lambda}$, and let $\mathcal{I}_T$ be an interval with $|\mathcal{I}_T| \in [ N_{\max}^{-2}/4,4N_{\max}^{-1}]$. Then the following estimate holds:
\begin{equation*}
\big| \iint_{\mathcal{I}_T \times \R} P_{N_1} R_{M_1} u_1 P_{N_2} R_{M_2} u_2 P_{N_3} R_{M_3} u_3 d\bar{x} dt \big| \lesssim T M_{\min}^{\frac{b}{2}} \prod_{i=1}^3 \| P_{N_i} R_{M_i} u_i \|_{V^2_{Ai} H^0_\Lambda}.
\end{equation*}
\end{proposition}
\begin{proof}
Under the assumptions of Proposition \ref{prop:ShorttimeBilinearEstimateAiry} on $K,N,I_j,\tilde{I}_j$ and $\text{dist}(\tilde{I}_1,\tilde{I}_2) \geq N_{\max}/16$ we record that for $u_1,u_2 \in V^2_{Ai} H^0_\Lambda$ the argument yields for $M_1 = M_2$
\begin{equation*}
\begin{split}
\| P_{I_1} R_{M_1} e^{t \partial_x^3} u_1 P_{I_2} R_{M_2} e^{t \partial_x^3} u_2 \|_{L^2_t( \mathcal{I}_T; \mathcal{L}^2_x)} &\lesssim T^{\frac{1}{2}} M^{\frac{b}{2}} \prod_{i=1}^2 \| P_{I_i} R_{M_i} u_i \|_{U^4_{Ai} H^0_\Lambda} \\
&\lesssim T^{\frac{1}{2}} M^{\frac{b}{2}} \prod_{i=1}^2 \| P_{I_i} R_{M_i} u_i \|_{V^2_{Ai} H^0_\Lambda}.
\end{split}
\end{equation*}
$V^2_{Ai}$ respects the almost orthogonal decomposition \eqref{eq:AlmostOrthogonalAiry} for $M_{\min} < M_{\max}$.
For $N_{\min} \leq N_{\max}/8$ the preceding estimate yields the claim by an almost orthogonal decomposition of the frequency height into cubes of size $M_{\min}$ and H\"older's inequality.
We turn to the case $N_{\min} \geq N_{\max}/8$, in which case we partition the frequency support into $I_i$ intervals of length $N_{\max}/32$ and can again suppose by almost orthogonality that $M_1=M_2=M_3$:
\begin{equation*}
\begin{split}
&\quad \big| \iint_{\mathcal{I}_T \times \R} P_{N_1} R_{M_1} u_1 P_{N_2} R_{M_2} u_2 P_{N_3} R_{M_3} u_3 d\bar{x} dt \big| \\
 &\leq \sum_{I_1,I_2,I_3} \big| \iint_{\mathcal{I}_T \times \R} P_{I_1} R_{M_1} u_1 P_{I_2} R_{M_2} u_2 P_{I_3} R_{M_3} u_3 d\bar{x} dt \big|.
\end{split}
\end{equation*}
It suffices to estimate a single expression in the above sum as claimed, as the number of non-trivial terms is clearly bounded.
We note that by convolution constraint
\begin{equation*}
\iint_{\mathcal{I}_T \times \R} P_{I_1} R_{M_1} u_1 P_{I_2} R_{M_2} u_2 P_{I_3} R_{M_3} u_3 d\bar{x} dt = 0,
\end{equation*}
unless there are $\xi_i \in I_i$ for which $\xi_1+\xi_2+\xi_3 = 0$. 
We consider a non-trivial expression. Suppose by symmetry that $I_1,I_2 \subseteq \R_{>0}$ and necessarily, $I_3 \subseteq \R_{<0}$. But $I_1+I_2$ is an interval of length $N_{\max}/16$ with $\inf (I_1+I_2) = \inf I_1 + \inf I_2$. Hence, $\text{dist}(-I_3,I_1), \text{dist}(-I_3,I_2) \geq N_{\max} / 16$, for which reason the bilinear estimate from Proposition \ref{prop:ShorttimeBilinearEstimateAiry} is applicable to $P_{I_2} R_{M_2} u_2 P_{I_3} R_{M_3} u_3$. The claim follows from H\"older's inequality.
\end{proof}

The following will be used to estimate boundary terms:
\begin{proposition}
\label{prop:AiryBoundaryTerms}
Let $N_i, R_j$, $i,j=1,2,3$ and $u_i \in V^2_{Ai} H^0_\Lambda$, and let $\mathcal{I}_T$ be an interval with $|\mathcal{I}_T| \leq N_{\max}^{-2}$. Then the following estimate holds:
\begin{equation*}
\big| \iint_{\mathcal{I}_T \times \R} P_{N_1} R_{M_1} u_1 P_{N_2} R_{M_2} u_2 P_{N_3} R_{M_3} u_3 d\bar{x} dt \big| \lesssim T N_{\min}^{\frac{1}{2}} M_{\min}^{\frac{b}{2}} \prod_{i=1}^3 \| P_{N_i} R_i u_i \|_{V^2_{Ai} H^0_\Lambda}.
\end{equation*}
\end{proposition}
\begin{proof}
By the above almost orthogonality arguments, we can suppose that $A_i=A_j$ for $i,j \in \{1,2,3\}$ and $A \in \{M,N\}$. Then the claim is immediate from integrating in time and Bernstein's inequality.
\end{proof}

\section{Short-time estimates}
\label{section:NonlinearEstimates}
\subsection{Short-time nonlinear estimates}

This section is devoted to the control of the nonlinearity in short-time function spaces. In the remainder of the section let $T \in (0,1]$. The main result reads as follows:
\begin{proposition}
\label{prop:NonlinearEstimate}
Let $r' \geq r > \frac{b}{2}$, $u_1 \in F^r(T)$, $u_2 \in F^{r'}(T)$. The following holds:
\begin{equation*}
\| \partial_x (u_1 u_2) \|_{\mathcal{N}^{r'}(T)} \lesssim T^{\frac{1}{4}} \| u_1 \|_{F^{r}(T)} \| u_2 \|_{F^{r'}(T)}.
\end{equation*}
For $r' > \frac{b}{2}$, $u_1 \in F^{r'}(T)$, $u_2 \in F^0(T)$ the following holds:
\begin{equation*}
\| \partial_x (u_1 u_2) \|_{\mathcal{N}^{0}(T)} \lesssim T^{\frac{1}{4}} \| u_1 \|_{F^{r'}(T)} \| u_2 \|_{F^{0}(T)}.
\end{equation*}
\end{proposition}
\begin{proof}
After plugging in the definition of the function spaces and height- and frequency-localization, it suffices to show the following estimate with \\ $|\mathcal{I}_{T_3}| = \min( T , N_3^{-\frac{3}{2}})$ and $|\mathcal{I}_{T_i}| \leq \min(T,N_i^{-\frac{3}{2}})$:
\begin{equation}
\label{eq:LocalizedNonlinearEstimate}
\begin{split}
&\quad \| P_{N_3} R_{M_3} \partial_x (P_{N_1} R_{M_1} u_1 P_{N_2} R_{M_2} u_2) \|_{DU_{Ai}^2(\mathcal{I}_{T_3}, \mathcal{L}^2_x)} \\
&\lesssim T^{\frac{1}{4}} M_{\min}^{\frac{b}{2}} \prod_{i=1}^2 \| P_{N_i} R_{M_i} u_i \|_{U^2_{Ai}(\mathcal{I}_{T_i}, \mathcal{L}^2_x)}.
\end{split}
\end{equation}
From this estimate the claim follows from dyadic summations.

We first handle the contribution of interactions with $N_3 \geq \max(N_1,N_2)/8$. In this case it suffices to consider a finite partition of $\mathcal{I}_{T_3}$ into intervals $\mathcal{I}_T$ of length $T \wedge N^{-3/2}_{\max}$:
\begin{equation*}
\begin{split}
&\quad \| P_{N_3} R_{M_3} \partial_x (P_{N_1} R_{M_1} u_1 P_{N_2} R_{M_2} u_2) \|_{DU^2_{Ai}(\mathcal{I}_{T},\mathcal{L}^2_x)} \\
 &= \sup_{\| v \|_{V^2_{Ai} \mathcal{L}^2_x} = 1} \big| \iint_{\mathcal{I}_{T} \times \R} P_{N_3} R_{M_3} \partial_x v P_{N_1} R_{M_1} u_1 P_{N_2} R_{M_2} u_2 d\overline{x} dt \big|.
\end{split}
\end{equation*}
For $N_3^{-2} \leq |\mathcal{I}_T| \leq N_3^{-\frac{3}{2}}$ we apply Proposition \ref{prop:TrilinearEstimate} to find 
\begin{equation*}
\begin{split}
&\quad \sup_{\| v \|_{V^2_{Ai} \mathcal{L}^2_x} = 1} \big| \iint_{\mathcal{I}_{T} \times \R} P_{N_3} R_{M_3} \partial_x v P_{N_1} R_{M_1} u_1 P_{N_2} R_{M_2} u_2 d\overline{x} dt \big| \\
&\lesssim N_3 |\mathcal{I}_{T_3}| M_{\min}^{\frac{b}{2}} \prod_{i=1}^2 \| P_{N_i} R_{M_i} u_i \|_{U^2_{Ai} H^0_\Lambda} \lesssim T^{\frac{1}{4}} M_{\min}^{\frac{b}{2}} \prod_{i=1}^2 \| P_{N_i} R_{M_i} u_i \|_{U^2_{Ai} H^0_\Lambda}.
\end{split}
\end{equation*}
For $N_3^{-2} \geq |\mathcal{I}_{T_3}|$ we can apply Proposition \ref{prop:AiryBoundaryTerms}, which yields the claim.

We turn to  $N_3 \leq \max( N_1,N_2)/8$. The difference to the previous cases is that the expression is merely localized to times $\min(N_3^{-3/2},T)$, but it is necessary to estimate the functions $P_{N_i} u_i$ on time-intervals of length $\min(N_1^{-3/2},T)$. We increase the time localization accordingly and then apply Proposition \ref{prop:TrilinearEstimate} or \ref{prop:AiryBoundaryTerms}. 

When $|\mathcal{I}_{T_3}| \sim N_3^{-\frac{3}{2}} \leq T$, we decompose $\mathcal{I}_{T_3}$ into $\big( \frac{N_1}{N_3} \big)^{\frac{3}{2}}$ intervals of length $N_1^{-\frac{3}{2}}$. Then we find by applying Proposition \ref{prop:TrilinearEstimate}
\begin{equation*}
\begin{split}
&\quad \sup_{\| v \|_{V^2_{KdV} \mathcal{L}^2_x} = 1} \big| \iint P_{N_3} R_{M_3} \partial_x v P_{N_1} R_{M_1} u_1 P_{N_2} R_{M_2} u_2 d\overline{x} dt \big| \\
&\lesssim N_3 N_1^{-\frac{3}{2}} \big( \frac{N_1}{N_3} \big)^{\frac{3}{2}} M_{\min}^{\frac{b}{2}} \prod_{i=1}^2 \| P_{N_i} R_{M_i} u_i \|_{U^2_{Ai}(\mathcal{I}_{T_i}; H^0_\Lambda)} \\
&\lesssim T^{\frac{1}{4}} M_{\min}^{\frac{b}{2}} \prod_{i=1}^2 \| P_{N_i} R_{M_i} u_i \|_{U^2_{Ai}(\mathcal{I}_{T_i}; H^0_\Lambda)}.
\end{split}
\end{equation*}

When $|\mathcal{I}_{T_3}| \sim T$ (because $N_3^{-\frac{3}{2}} \geq T$) and $N_1^{-\frac{3}{2}} \leq T$, we find by partitioning into intervals of length $N_1^{-3/2}$:
\begin{equation*}
\begin{split}
&\quad \sup_{\| v \|_{V^2_{KdV} H^0_\Lambda} = 1} \big| \iint P_{N_3} R_{M_3} \partial_x v P_{N_1} R_{M_1} u_1 P_{N_2} R_{M_2} u_2 d\overline{x} dt \big| \\
&\lesssim T N_1^{\frac{3}{2}} N_1^{-\frac{3}{2}} M_{\min}^{\frac{b}{2}} \prod_{i=1}^2 \| P_{N_i} R_{M_i} u_i \|_{U^2_{Ai}(\mathcal{I}_{T_i}; H^0_\Lambda)}.
\end{split}
\end{equation*}

Finally, when $|\mathcal{I}_{T_3}| \sim T \leq N_1^{-\frac{3}{2}}$, there is no further partitioning into smaller intervals. For $T \in (N_1^{-2},N_1^{-\frac{3}{2}}]$ we apply Proposition \ref{prop:TrilinearEstimate} and for $T \leq N_1^{-2}$ we apply Proposition \ref{prop:AiryBoundaryTerms}. In both cases the claim is immediate.
This finishes the proof.
\end{proof}

\subsection{Short-time energy estimates}

We turn to the proof of energy estimates. For solutions to \eqref{eq:KdVIntro} we shall take advantage of the conservative derivative nonlinearity and integrate by parts. We prove:
\begin{proposition}
\label{prop:EnergyEstimates}
Let $u \in C([0,T],H^{10\nu}_\Lambda) \cap C^1((0,T),H^{10 \nu-3}_{\Lambda})$ be a solution to \eqref{eq:KdVIntro}. The following estimate holds for $r' \geq r > \frac{b}{2} + 1$:
\begin{equation}
\label{eq:EnergyEstimateSolution}
\| u \|^2_{E^r(T)} \leq \| u(0) \|^2_{\mathcal{H}^r_\Lambda} + C T^{\frac{1}{2}} \|u \|^2_{F^r(T)} \| u \|_{F^{r'}(T)}.
\end{equation}
\end{proposition}
\begin{proof}
We write
\begin{equation*}
\| P_N u(t) \|^2_{H^r_{\Lambda}} = \sum_{M \geq N} M^{2 r} \| P_N R_M u \|^2_{H^0_\Lambda},
\end{equation*}
then we invoke the fundamental theorem of calculus
\begin{equation*}
\| P_N R_M u \|^2_{\mathcal{L}^2} = \| P_N R_M u(0) \|^2_{\mathcal{L}^2} + 2 \int_0^t \int_{\R} P_N R_M u P_N R_M \partial_x (u \cdot u) d\overline{x} ds.
\end{equation*}
In the following we suppose that $N \geq 4$ as the estimate of the remainder terms is straight-forward. We have the following paraproduct decomposition:
\begin{equation*}
P_N \big( \sum_{N_1,N_2} P_{N_1} u P_{N_2} u \big) = 2 P_N \big( \sum_{\substack{N_1 \in [N/4,4N], \\ N_2 \leq N/8}} P_{N_1} u P_{N_2} u \big) + P_N \big( \sum_{N_1,N_2 \gtrsim N} P_{N_1} u P_{N_2} u \big).
\end{equation*}
We carry out another paraproduct decomposition and denote resulting terms by
\begin{equation}
\label{eq:ParaproductDecomposition}
P_N R_M \partial_x (u^2) = P_N R_M \partial_x \big[ 2(P_{\ll N} R_{\ll M} u \cdot u) + 2 (P_{\ll N} u R_{\ll M} u) + (P_{\gtrsim N} u P_{\gtrsim N} u) \big].
\end{equation}
For an estimate of the first term in \eqref{eq:ParaproductDecomposition}
we can decompose
\begin{equation}
\label{eq:ParaproductDecompositionA}
\begin{split}
&\quad \iint P_N R_M u [P_N R_M \partial_x (P_{\ll N} R_{\ll M} u \cdot u)] \\
&= \iint P_N R_M u \partial_x (P_{\ll N} R_{\ll M} u \cdot P_N R_M u) d\bar{x} dt \\
&\; + \iint P_N R_M u \partial_x (P_N R_M [(P_{\ll N} R_{\ll M} u) \cdot u) - P_{\ll N} R_{\ll M} u \cdot P_N R_M u) ] d \bar{x} dt.
\end{split} 
\end{equation}
In the first term in \eqref{eq:ParaproductDecompositionA} we can integrate by parts such that the derivative can be assigned to $P_{\ll N} R_{\ll M} u$. For the second term, we note that by anisotropy the commutator estimate will involve also the total regularity $M$.
We turn to an estimate of the first term in \eqref{eq:ParaproductDecomposition} by \eqref{eq:ParaproductDecompositionA}. Write with $N' \leq N/4$, $M' \leq M/4$
\begin{equation*}
\begin{split}
P_N R_M ( P_{N'} R_{M'} u \cdot u) &= P_{N'} R_{M'} u P_N R_M u \\
&\quad + [ P_N R_M ( P_{N'} R_{M'} u \cdot u) - P_{N'} R_{M'} u P_N R_M u ].
\end{split}
\end{equation*}
The contribution of the first expression is easily handled by integration by parts:
\begin{equation*}
\begin{split}
&\quad \iint P_N R_M u \partial_x (P_{N'} R_{M'} u P_N R_M u) d\bar{x} ds \\
&= - \frac{1}{2} \iint P_N R_M u \partial_x (P_{N'} R_{M'} u) P_N R_M u d\bar{x} ds.
\end{split}
\end{equation*}
We localize time to intervals of length $\min(T, N^{-\frac{3}{2}})$ and first treat the main contribution $T \geq N^{-\frac{3}{2}}$. Clearly, we can decompose $[0,T]$ into $\sim T N^{\frac{3}{2}}$ intervals $\mathcal{I}_T$ of length $N^{-\frac{3}{2}}$. We can estimate this contribution by Proposition \ref{prop:TrilinearEstimate} which recovers the factor of $N \geq N'$:
\begin{equation}
\label{eq:AuxEnergy}
\begin{split}
&\quad \sum_{N' \leq M' \leq M} \big| \iint_{\mathcal{I}_T \times \R} P_N R_M u (\partial_x P_{N'} R_{M'} u ) P_N R_M u \big| \\
&\lesssim N^{-\frac{3}{2}} \sum_{N' \lesssim M' \lesssim M} N' (M')^{\frac{b}{2}} \| P_{N'} R_{M'} u \|_{V^2_{Ai}(\mathcal{I}_T)} \| P_N R_M u \|^2_{V^2_{Ai}(\mathcal{I}_T)}.
\end{split}
\end{equation}
We carry out the sum in $M'$ to find for $r> \frac{b}{2}+1$:
\begin{equation*}
\begin{split}
\sum_{N' \lesssim M' \lesssim M} N' (M')^{\frac{b}{2}} \| P_{N'} R_{M'} u \|_{V^2_{Ai}(\mathcal{I}_T,H^0_{\Lambda})}  &\lesssim \sum_{N' \lesssim N} (N')^{-\varepsilon} \| P_{N'} u \|_{V^2_{Ai}(\mathcal{I}_T,H^r_{\Lambda})} \\
&\lesssim \| u \|_{F^r(T)}.
\end{split}
\end{equation*}
Carrying out the sum over $M \geq N$, then over $N$ in \eqref{eq:AuxEnergy} and taking into account the decomposition into $T N^{\frac{3}{2}}$ estimates this contribution in terms of $T \| u \|^2_{F^r(T)} \| u \|_{F^{r'}(T)}$. For $T \in [N^{-2},N^{-\frac{3}{2}}]$ we can still argue like above, but do not require anymore the partition into intervals of length $N^{-\frac{3}{2}}$. This contributes as $T^{\frac{1}{2}} \| u \|_{F^r(T)}^2 \| u \|_{F^{r'}(T)}$. In case $T \leq N^{-2}$, following along the above lines and applying Proposition \ref{prop:AiryBoundaryTerms} instead of Proposition \ref{prop:TrilinearEstimate}, yields again the estimate $T^{\frac{1}{2}} \| u \|_{F^r(T)}^2 \| u \|_{F^{r'}(T)}$.

\smallskip

For the second expression in \eqref{eq:ParaproductDecompositionA} we obtain the multiplier in Fourier space
\begin{equation*}
\begin{split}
&\quad \chi_{N'}(k_{1} \cdot \omega) \chi_{M'}(k_1) \big( \chi_N((k_1 + k_2) \cdot \omega ) \chi_M(k_1+k_2) - \chi_N( k_2 \cdot \omega) \chi_M(k_2) \big) \\
&= \chi_{N'}(k_1 \cdot \omega) \chi_{M'}(k_1) (\chi_N(( k_1 + k_2 ) \cdot \omega) - \chi_N(k_2 \cdot \omega)) \chi_M(k_1 + k_2) \\
&\quad + \chi_{N'}(k_1 \cdot \omega) \chi_{M'}(k_1) \chi_N(k_2 \cdot \omega) ( \chi_M(k_1+k_2) - \chi_M(k_2) ) .
\end{split}
\end{equation*}
Accordingly, we split the contribution as
\begin{equation*}
\partial_x [ P_N R_M ( P_{N'} R_{M'} u \cdot u) - P_{N'} R_{M'} u P_N R_M u ] = CI + CII.
\end{equation*}

To estimate the expressions, we want to use that the multipliers have a favorable modulus by the mean-value theorem. We have
\begin{equation}
\label{eq:CommI}
\big| \chi_{N'}(k_1 \cdot \omega) \chi_{M'}(k_1) \chi_N( ( k_1 + k_2 ) \cdot \omega) - \chi_N(k_2 \cdot \omega) \chi_M(k_1+k_2) \big| \lesssim \frac{N'}{N}.
\end{equation}
For the second term we find
\begin{equation}
\label{eq:CommII}
\big| \chi_{N'}(k_1 \cdot \omega) \chi_{M'}(k_1) \chi_N(k_2 \cdot \omega) \big[ \chi_M(k_1+k_2) - \chi_M(k_2) \big] \big| \lesssim \frac{M'}{M}.
\end{equation}
To conclude the estimate, we carry out a Fourier series expansion of the multipliers, cf. \cite[Remark~5.9]{KimSchippa2021}. The frequency-localization is maintained and the function spaces are translation-invariant, for which reason the Fourier series expansion is permissible.

With the size estimate at hand, we turn to short-time estimates. For the first term we can argue precisely like above. We localize time to intervals of length $\min(T, N^{-\frac{3}{2})}$ and apply Proposition \ref{prop:TrilinearEstimate} for $T \geq N^{-2}$, resp. \ref{prop:AiryBoundaryTerms} for $T \leq N^{-2}$.

The estimate of \eqref{eq:CommII} reads different. Localizing time to intervals $\mathcal{I}_T$ to \\ $\min(T, N^{-\frac{3}{2}})$, taking into account the derivative loss, commutator estimate \eqref{eq:CommII} and Proposition \ref{prop:TrilinearEstimate}, we find for $T \geq N^{-\frac{3}{2}}$ such that $|\mathcal{I}_T| \sim N^{-\frac{3}{2}}$:
\begin{equation*}
\begin{split}
&\quad \big| \iint_{\mathcal{I}_T \times \R} P_N R_M u (CII) d\bar{x} dt \big| \\
 &\lesssim N \frac{M'}{M} N^{-\frac{3}{2}} (M')^{\frac{b}{2}} \| P_N R_M u \|_{V^2_{Ai}(\mathcal{I}_T)} \| P_{N'} R_{M'} u \|_{V^2_{Ai}(\mathcal{I}_T)} \| \tilde{P}_N \tilde{R}_M u \|_{V^2_{Ai}(\mathcal{I}_T)} \\
&\lesssim \frac{1}{N^{\frac{1}{2}} M} (M')^{\frac{b}{2}+1} \| P_N R_M u \|_{V^2_{Ai}(\mathcal{I}_T)} \| P_{N'} R_{M'} u \|_{V^2_{Ai}(\mathcal{I}_T)} \| \tilde{P}_N \tilde{R}_M u \|_{V^2_{Ai}(\mathcal{I}_T)}.
\end{split}
\end{equation*}
Above we let $\tilde{A}_K = A_{K/4} + A_{K/2} + A_K + A_{2K} + A_{4K}$.
This estimate is inferior to the previous one, but can be summed up to the same total regularity $r > \frac{b}{2} + 1$, and taking into account the small time intervals. For $T \leq N^{-\frac{3}{2}}$ the necessary modifications are like above. The estimate of the first term in \eqref{eq:ParaproductDecomposition} is complete.

\smallskip

We turn to the estimate of the second term in \eqref{eq:ParaproductDecomposition}, which is given by 
\begin{equation*}
\sum_{\substack{M' \leq M, \\ N' \leq N}} \iint P_N R_M u \partial_x (P_{N'} \tilde{R}_M u \cdot \tilde{P}_N R_{M'} u) d\bar{x} dt, \quad M' \leq M, \; N' \leq N.
\end{equation*}
Integration by parts to the assign the derivative to the lowest tangential frequency is not possible. Observe that $M_{\min} \gtrsim N$, for which reason high tangential frequencies are still dominated by small total frequencies. By localizing time to intervals $\mathcal{I}_T$ of length $\min(T,N^{-\frac{3}{2}})$ and applying Proposition \ref{prop:TrilinearEstimate} or Proposition \ref{prop:AiryBoundaryTerms} depending on $|\mathcal{I}_T|$ like above, estimates this contribution.

In the third term in \eqref{eq:ParaproductDecomposition} the derivative is on the factor with low tangential frequencies. This can be estimated like above and completes the proof.
\end{proof}


\begin{proposition}
\label{prop:EnergyEstimateDifferences}
Let $v \in C([0,T],H^{10\nu}_\Lambda) \cap C^1((0,T),H^{10 \nu-3}_{\Lambda}) \cap F^0(T)$ be a solution to
\begin{equation}
\label{eq:DifferenceEquation}
\partial_t v +  \partial_x^3 v = \partial_x (v u).
\end{equation}
Then it holds for $r> \frac{b}{2}+1$:
\begin{equation}
\label{eq:EnergyEstimateDifferences}
\| v \|^2_{E^0(T)} \lesssim \| v(0) \|^2_{H^0_{\Lambda}} + T^{\frac{1}{2}} \| v \|^2_{F^0(T)} \| u \|_{F^r(T)}.
\end{equation}
\end{proposition}
\begin{proof}
The main difference with the previous energy estimate is the decreased symmetry of the expression, which is accounted for by an estimate at $\mathcal{L}^2$-regularity.

We start out like in the previous proposition invoking the fundamental theorem of calculus with $M \geq N \geq 4$:
\begin{equation*}
\| P_N R_M v(t) \|^2_{\mathcal{L}^2_x} = \| P_N R_M v(0) \|^2_{\mathcal{L}^2_x} + 2 \int_0^t \int_{\R} P_N R_M v \partial_x P_N R_M (v u) d\bar{x} ds.
\end{equation*}
We use a paraproduct decomposition
\begin{equation*}
\begin{split}
P_N R_M (u v) &= P_N R_M (P_{\ll N} R_{\ll M} u \cdot v) + P_N R_M (u \cdot P_{\ll N} R_{\ll M} v) \\
&\quad + P_N R_M (P_{\ll N} u R_{\ll M} v + R_{\ll M} u P_{\ll N} v) + P_N R_M (P_{\gg N} u P_{\gg N} v)).
\end{split}
\end{equation*}
The overall strategy is like above, so we shall be brief. The time interval $[0,T]$ is decomposed into intervals $\mathcal{I}_T$ of length $\min(T,N_{\max}^{-\frac{3}{2}})$.

The first expression is estimated like in the previous proposition by integration by parts and commutator estimates. The only difference is a shift in regularity, which is readily accounted for.

The second expression is not amenable to integration by parts:
\begin{equation*}
\sum_{\substack{N' \leq N, \\ M' \leq M}} \iint_{\mathcal{I}_T \times \R} P_N R_M v \partial_x (\tilde{P}_N \tilde{R}_M u P_{N'} R_{M'} v). 
\end{equation*}
For $|\mathcal{I}_T| \in [N^{-2}, N^{-\frac{3}{2}}]$ an application of Proposition \ref{prop:TrilinearEstimate} yields for $N' \leq N$, $M' \leq M$
\begin{equation*}
\begin{split}
&\quad \big| \iint_{\mathcal{I}_T \times \R} P_N R_M v \partial_x (\tilde{P}_N \tilde{R}_M u P_{N'} R_{M'} v ) d \bar{x} dt \big| \\
&\lesssim N |\mathcal{I}_T| M_{\min}^{\frac{b}{2}} \| P_N R_M v \|_{V^2_{Ai}(\mathcal{I}_T)} \| \tilde{P}_N \tilde{R}_M u \|_{V^2_{Ai}(\mathcal{I}_T)} \| P_{N'} R_{M'} v \|_{V^2_{Ai}(\mathcal{I}_T)}.
\end{split}
\end{equation*}
Since $u$ will be estimated at high regularity, this is acceptable. For $|\mathcal{I}_T| \leq N^{-2}$ we apply Proposition \ref{prop:AiryBoundaryTerms} instead. In both cases summation is straight-forward.

Similarly, the third and fourth term are estimated. For the third term note that
\begin{equation*}
\begin{split}
 &\iint_{\mathcal{I}_T \times \R} P_N R_M v \partial_x P_N R_M (P_{\ll N} \tilde{R}_M u \tilde{P}_N R_{\ll M} v) d\bar{x} dt, \\
 &\quad \iint_{\mathcal{I}_T \times \R} P_N R_M v \partial_x P_N R_M (\tilde{P}_N R_{\ll M} u P_{\ll N} v \tilde{R}_M v) d\bar{x} dt.
 \end{split}
\end{equation*}
the minimal total frequency dominate the high transverse frequency, for which reason an estimate without integration by parts and commutator arguments is possible.

In the fourth term note that the derivative acts already on the low frequency. The straight-forward details on the summations are omitted.

%
%

\end{proof}

\section{Low regularity well-posedness for the Benjamin-Ono equation}
\label{section:BenjaminOno}

In this section we show Theorem \ref{thm:QPBenjaminOno}, which is concerned with low-regularity well-posedness of the Benjamin-Ono equation with small quasi-periodic initial data:
\begin{equation}
\label{eq:BenjaminOnoProof}
\left\{ \begin{array}{cl}
\partial_t u + \mathcal{H} \partial_x^2 u &= u \partial_x u, \quad (t,x) \in \R \times \R, \\
u(0) &= u_0 \in H^s_{\Lambda}.
\end{array} \right.
\end{equation}
 Starting point is again the data-to-solution mapping $S_T^{\infty} : H^{10 \nu}_{\Lambda} \to C([0,T],H^{10 \nu}_{\Lambda})$ provided by the energy arguments in \cite{AitzhanAmbrose2024}. We sketch the proof of the following:
\begin{theorem}
For $s > \frac{\nu+1}{2}$ there is $c > 0$ such that the data-to-solution mapping $S_T^\infty$ extends uniquely to a continuous map $S_1^s: B_{H^s_\Lambda}(0,c) \to C([0,1],H^s_\Lambda)$.
\end{theorem}

\subsection{Function spaces and multilinear estimates}
The short-time spaces are defined analogous to Section \ref{section:Prelim} replacing the Airy propagator with $e^{t \mathcal{H} \partial_x^2}$. Importantly, we choose the frequency-dependent time localization $T=T(N)=N^{-1}$ like in \cite{ShorttimeFourierTransformRestriction}, in which the periodic case was treated in detail. We denote the resulting short-time function spaces with $F^s_{BO}(T)$, $\mathcal{N}^s_{BO}(T)$.

The bilinear square function estimate proved in Theorem \ref{thm:SFBO} yields the following analog of Proposition \ref{prop:ShorttimeBilinearEstimateAiry}:
\begin{proposition}
\label{prop:ShorttimeBilinearBO}
Let $K \leq 2N$ and $I_j \in (-3N,3N)$, $j=1,2$ intervals of length $K$. Let $\tilde{I}_j = \{ x \in \R_{>0} : x \in I_j \vee -x \in I_j\}$ and assume that $\text{dist}(\tilde{I}_1,\tilde{I}_2) \geq N/16$. Assume $u_1,u_2 \in H^0_{\Lambda}$. Let $N^{-2} \leq T \leq N^{-1}$. The following holds:
\begin{equation*}
\| P_{I_1} R_{M_1} e^{t \mathcal{H} \partial_x^2} u_1 P_{I_2} R_{M_2} e^{t \mathcal{H} \partial_x^2} u_2 \|_{L^2_t([0,T],\mathcal{L}^2_x)} \lesssim N^{-\frac{1}{2}} M_{\min}^{\frac{\nu-1}{2}} \| u_1 \|_{H^0_{\Lambda}} \| u_2 \|_{H^0_{\Lambda}}.
\end{equation*}
\end{proposition}
\begin{proof}
By an almost orthogonal decomposition we can suppose that $M_1 = M_2$. We rescale $t \to N^2 t$, $x \to Nx$, $\xi \to \xi/N$ which restricts frequencies to $(-3,3)$ with unit separation:
\begin{equation*}
\begin{split}
&\quad \| P_{I_1} R_{M_1} e^{t \mathcal{H} \partial_x^2} u_1 P_{I_2} R_{M_2} e^{t \mathcal{H} \partial_x^2} u_2 \|_{L^2_t([0,T],\mathcal{L}^2_x)} \\
&= \| P_{I_1'} R_{M_1'} e^{t \mathcal{H} \partial_x^2} u_1' P_{I_2'} R_{M_2'} e^{t \mathcal{H} \partial_x^2} u_2' \|_{L^2_t([0,TN^2],\mathcal{L}^2_x)}.
\end{split}
\end{equation*}
After the approximation argument like in the proof of Proposition \ref{prop:ShorttimeBilinearEstimateAiry}, this is amenable to Theorem \ref{thm:SFBO} with $\delta= (TN^2)^{-1}$. We find
\begin{equation*}
\begin{split}
&\quad \| P_{I_1} R_{M_1} e^{t \mathcal{H} \partial_x^2} u_1 P_{I_2} R_{M_2} e^{t \mathcal{H} \partial_x^2} u_2 \|_{L^2_t([0,T],\mathcal{L}^2_x)} \\
&\lesssim \big\| \big( \sum_{\theta_1 \in \bar{\mathbb{I}}_1} |P_{I_1} R_{M_1} e^{t \mathcal{H} \partial_x^2} P_{\theta_1} u_1|^2 \big)^{\frac{1}{2}} \big( \sum_{\theta_2 \in \bar{\mathbb{I}}_1} |P_{I_2} R_{M_2} e^{t \mathcal{H} \partial_x^2} P_{\theta_2} u_2|^2 \big)^{\frac{1}{2}} \big\|_{L^2_t(w_T,\mathcal{L}^2_x)}
\end{split}
\end{equation*}
with $\bar{\mathbb{I}}_1 = \{ T^{-1} N^{-1} (k,k+1] : k \in \Z \}$ and $P_{\theta_i}$ denotes the corresponding frequency projection. We estimate by H\"older's and Bernstein's inequality:
\begin{equation*}
\begin{split}
&\quad \| P_{I_i} R_{M_i} e^{t \mathcal{H} \partial_x^2} P_{\theta_i} u_i \|_{L^4_t(w_T, \mathcal{L}^4_x))} \\
&\lesssim T^{\frac{1}{4}} (T^{-1} N^{-1} )^{\frac{1}{4}} M_{\min}^{\frac{\nu-1}{4}} \| P_{I_i} R_{M_i} P_{\theta_i} u_i \|_{H^0_{\Lambda}}.
\end{split}
\end{equation*}
The frequency-localized pieces can be summed up by almost orthogonality.
\end{proof}

By the above, we obtain the following variant of Proposition \ref{prop:TrilinearEstimate}:
\begin{proposition}
\label{prop:TrilinearEstimateProp}
Let $N_i,M_i$, $i=1,2,3$ and $u_i \in V^2_{BO} H^0_{\Lambda}$, and let $\mathcal{I}_T$ be an interval of length $|\mathcal{I}_T| = T \in [N_{\max}^{-2}/4,4N_{\max}^{-1}]$. Then the following holds:
\begin{equation*}
\big| \iint_{\mathcal{I}_T \times \R} \prod_{i=1}^3 P_{N_i} R_{M_i} u_i d \bar{x} dt \big| \lesssim T^{\frac{1}{2}} N_{\max}^{-\frac{1}{2}} M_{\min}^{\frac{b}{2}} \prod_{i=1}^3 \|P_{N_i} R_{M_i} u_i \|_{V^2_{BO} H^0_{\Lambda}}.
\end{equation*}
\end{proposition}
Finally, we remark that a version of Proposition \ref{prop:AiryBoundaryTerms} for function spaces adapted to $e^{t \mathcal{H} \partial_x^2}$ is immediate. 

\subsection{Short-time estimates}

With the multilinear estimates from the previous subsection at hand, the following short-time bilinear estimates can be proved like in Section \ref{section:NonlinearEstimates}. The key difference is a different choice of frequency-localization $T=T(N)=N^{-1}$, which is the same as in \cite[Section~3]{ShorttimeFourierTransformRestriction}. Let $T \in (0,1]$ for the remainder of this section.
\begin{proposition}
Let $r' \geq r > \frac{b}{2}$, $u_1 \in F_{BO}^r(T)$, $u_2 \in F^{r'}_{BO}(T)$. The following holds:
\begin{equation*}
\| \partial_x (u_1 u_2) \|_{\mathcal{N}^{r'}_{BO}(T)} \lesssim \| u_1 \|_{F^r(T)} \| u_2 \|_{F^{r'}(T)}.
\end{equation*}
For $r' > \frac{b}{2}$, $u_1 \in F_{BO}^{r'}(T)$, $u_2 \in F^{0}_{BO}(T)$ the following holds:
\begin{equation*}
\| \partial_x (u_1 u_2) \|_{\mathcal{N}^0_{BO}(T)} \lesssim \| u_1 \|_{F^{r'}_{BO}(T)} \| u_2 \|_{F^0_{BO}(T)}.
\end{equation*}
\end{proposition}
The worse dependence on the frequency-dependent time in Proposition \ref{prop:TrilinearEstimateProp} compared to Proposition \ref{prop:TrilinearEstimate} suggests to work with the frequency-dependent time localization $T=T(N)=N^{-1}$, which is further explained in \cite[Section~3]{ShorttimeFourierTransformRestriction}, and leads to worse estimates compared to Proposition \ref{prop:NonlinearEstimate} as these do not depend on $T$ anymore.

Correspondingly, we obtain the following short-time energy estimates:
\begin{proposition}
Let $u \in C([0,T],H^{10\nu}_\Lambda) \cap C^1((0,T),H^{10 \nu-2}_{\Lambda}) \cap F^r(T) $ be a solution to \eqref{eq:BenjaminOnoProof}. The following estimate holds for $r' \geq r > \frac{b}{2} + 1$:
\begin{equation*}
\| u \|^2_{E^r(T)} \leq \| u(0) \|^2_{\mathcal{H}^r_\Lambda} + C \|u \|^2_{F^r(T)} \| u \|_{F^{r'}(T)}.
\end{equation*}

Let $v \in C([0,T],H^{10\nu}_\Lambda) \cap C^1((0,T),H^{10 \nu-2}_{\Lambda}) \cap F_{BO}^0(T)$ be a solution to
\begin{equation*}
\partial_t v +  \mathcal{H} \partial_x^2 v = \partial_x (v u).
\end{equation*}
Then it holds for $r> \frac{b}{2}+1$:
\begin{equation*}
\| v \|^2_{E^0(T)} \lesssim \| v(0) \|^2_{H^0_{\Lambda}} + \| v \|^2_{F_{BO}^0(T)} \| u \|_{F_{BO}^r(T)}.
\end{equation*}
\end{proposition}

\subsection{Conclusion of the local well-posedness result}

We obtain the following set of estimates for sufficiently regular solutions on $[0,T]$, $T \in (0,1]$ provided that $r' \geq r > \frac{\nu+1}{2}$:
\begin{equation*}
\left\{ \begin{array}{cl}
\| u \|_{F_{BO}^{r'}(T)} &\lesssim \| u \|_{E^{r'}(T)} + \| \partial_x(u^2) \|_{\mathcal{N}_{BO}^{r'}(T)}, \\
\| \partial_x(u^2) \|_{\mathcal{N}_{BO}^{r'}(T)} &\lesssim \| u \|_{F_{BO}^{r}(T)} \| u \|_{F_{BO}^{r'}(T)}, \\
\| u \|^2_{E^{r'}(T)} &\lesssim \| u_0 \|_{H^{r'}_{\Lambda}}^2 + \| u \|_{F_{BO}^{r'}(T)}^2 \| u \|_{F_{BO}^r(T)}.
\end{array} \right.
\end{equation*}
By the persistence of regularity and bootstrap, this allows us to construct solutions in $H^s_{\Lambda}$ for $s>\frac{\nu+1}{2}$ on $[0,1]$, provided that the initial data are sufficiently small.

Next, we use the following set of estimates for differences of sufficiently regular solutions $v = u_1-u_2$ with $r > \frac{\nu+1}{2}$ and $\| u_1(0) \|_{H^r_{\Lambda}} + \| u_2(0) \|_{H^r_{\Lambda}} \leq c$:
\begin{equation*}
\left\{ \begin{array}{cl}
\| v \|_{F_{BO}^{0}(T)} &\lesssim \| v \|_{E^{0}(T)} + \| \partial_x(v(u_1+u_2)) \|_{\mathcal{N}_{BO}^{0}(T)}, \\
\| \partial_x(v(u_1+u_2)) \|_{\mathcal{N}_{BO}^{0}(T)} &\lesssim \| v \|_{F_{BO}^{0}(T)} (\| u_1 \|_{F_{BO}^{r}(T)} + \| u_2 \|_{F^r_{BO}(T)}), \\
\| v \|^2_{E^{0}(T)} &\lesssim \| v(0) \|_{H^{0}_{\Lambda}}^2 + \| v \|_{F_{BO}^{0}(T)}^2 ( \| u_1 \|_{F_{BO}^r(T)} + \| u_2 \|_{F_{BO}^r(T)}).
\end{array} \right.
\end{equation*}
This yields the already familiar Lipschitz-continuous dependence of the difference on the initial data at higher regularity. The frequency-envelope argument finishes the proof of continuous dependence.

$\hfill \Box$

\section*{Acknowledgements}

Financial support by the Humboldt foundation through a Feodor-Lynen return fellowship is gratefully acknowledged.

\end{document}